\documentclass[12pt]{article}

\usepackage{fullpage}
\usepackage{amsmath, amsthm, amssymb, enumerate}
\usepackage{color}
\usepackage{datetime}
\usepackage{fancyhdr}

\begin{document}
	
\newtheorem{Theorem}{Theorem}[section]
\newtheorem{Corollary}[Theorem]{Corollary}
\newtheorem{Proposition}[Theorem]{Proposition}
\newtheorem{Lemma}[Theorem]{Lemma}
\newtheorem{Assumptions}[Theorem]{Assumptions}
\newtheorem{definition}{Definition}[section]

\theoremstyle{definition}
\newtheorem{Remark}[Theorem]{Remark}
\newtheorem{example}[Theorem]{Example}


\title{{\bf Generalized exponentially bounded integrated semigroups }}

\author{
Marko  Kosti\'c\thanks{Faculty of Technical Sciences, University of Novi Sad, Serbia,  marco.s@verat.net} ,
Stevan Pilipovi\'c\thanks{Faculty of Sciences, University of Novi Sad, Serbia,  stevan.pilipovic@dmi.uns.ac.rs} ,
Milica \v Zigi\'c\thanks{Faculty of Sciences, University of Novi Sad, Serbia, milica.zigic@dmi.uns.ac.rs}
	} 
\maketitle
\date{}

\bigskip

\begin{abstract}
The main subject of this paper is the analysis of  sequences of exponentially bounded integrated semigroups which are related to Cauchy problems
\begin{equation}\label{jed}
\frac{\partial}{\partial t}u(t,x)-a(D)u(t,x)=f(t,x), \quad u(0,x)=u_0(x), \quad t\geq 0, \  x\in \mathbb R^d,
\end{equation}
with a distributional initial data $u_0$ and a distributional right hand side $f$ through a sequence of equations with regularized $u_0$ and $f$ and a sequence of (pseudo) differential operators $a_n(D)$ instead of $a(D)$. Comparison of sequences of infinitesimal generators and the determination of corresponding sequences of integrated semigroups are the main subject of the paper. For this purpose, we introduce association, the relation of equivalence for  infinitesimal generators on one side and the corresponding relations of equivalence of integrated semigroups on another side.  The order of involved assumptions on generators essentially characterize the mutual dependence of sequences of infinitesimal generators and the corresponding sequences of integrated semigroups.
\end{abstract}



\maketitle

\section{Introduction}
This paper aims to provide an approach to the sequences of infinitesimal generators and the corresponding sequences of integrated semigroups, usually obtained through a process of regularization,  as a framework for solving singular Cauchy problems within spaces of generalized functions.

General theory of integrated semigroups, introduced by Arendt \cite{A 1987}, was  already  stated in a large number of monographs.  We refer to a fundamental monograph \cite{ABHN 2011}  and references therein  for the historical background. For the applications, especially  in population biology and population persistence, we refer to   \cite{Tim2} and  \cite{MR}. The authors of quoted monographs are the leading ones in the field with a plenty of strong papers which can be found in the bibliography of these monographs. Actually, contributions to the theory of integrated semigroups were given by many excellent  papers which can be easily found.  We do not mention them because it is very likely that many important works will be omitted. Here we mention that one of coauthors has written several papers and monographs related to various kinds of semigroups. We refer to \cite{MK 11}, \cite{MK 15} and references therein.

Concerning generalized  $C_0$-semigroups through regularization, we refer to  \cite{NPR},  where,  in the frame of Colombeau theory \cite{C1984} of generalized function algebras, were discussed relations between nets of infinitesimal generators and the corresponding nets of $C_0$-semigroups, with applications to a certain class of nonlinear wave equations. In relation to \cite{NPR}, our approach in this paper is different; instead of nets, we simplify the exposition using sequences as a main tool and instead of technically more complex definitions of Colombeau theory we directly introduce sequences of solutions of regularized Cauchy problems with strong singularities.  Roughly speaking, our results   are related to approximations of infinitesimal generators and the corresponding integrated semigroups, through the sequences of such operators and integrated semigroups. 

We only deal with one time integrated semigroups although it is clear that  the paper can be extended to $k$-times integrated semigroups. In this way, we avoid the distributional semigroups for which we know that every one of them is a $k$-times integrated semigroup for a certain $k$ (cf. \cite{ku}, \cite{w}). Our approach is motivated by revisiting well known results for one time integrated semigroups which correspond to infinitesimal generators given as Fourier multipliers with symbols of the class $S^m_{1,0},$ $m\in\mathbb{N},$ on the Lebesgue space $L^p(\mathbb R^d),$ $p\geq1$ and those which correspond to symbols $i|\xi|^m, m\in \mathbb N$ given in  Section 8.3 of \cite{ABHN 2011}, see also \cite{H1}, \cite{H2}.  

In the main part of the paper we  analyse the relations between sequences of infinitesimal generators and corresponding sequences of integrated  semigroups in a sense that a certain perturbation of a sequence of infinitesimal generators  results by a perturbation of the corresponding sequence of integrated semigroups which we estimate and classify.   This is done by the introduction of associated sequences, as in the algebraic theory of generalized functions, cf. \cite{C1984}, \cite{MMMR}. 

The paper is organized as follows. Notation is the standard one  for the real  numbers, as well as for the Lebesgue $L^p$-spaces, Schwartz test function and distribution spaces. In Introduction, Subsection \ref{ss not}, we introduce sequence spaces over the pivot Banach space $X$ as a framework for the further investigations. The moderate growth of involved sequences is the essential assumption of all sequence spaces considered in the paper. Section \ref{s 2} is related to sequences of closed linear operators defined on $X$ (Subsection \ref{ss sog}) since they are under suitable conditions infinitesimal generators  of sequences of  exponentially bounded integrated semigroups (Subsection  \ref{ss geis}). For this purpose we impose additional conditions on operators and call them sequences of  infinitesimal generators. In Section \ref{examples}, we revisit some examples presented in the monograph \cite{ABHN 2011}, given in Section 8.3, related to a class of pseudo-differential operators. This section illustrates just a few possibilities for the applications in finding a solution to \eqref{jed}, in the form of a sequence,  although  the singular equation from which we started does not have a solution in a classical analysis setting.  If a problem \eqref{jed} has a classical solution, one must have that obtained sequence of solutions, also called very weak solution (cf.  \cite{GR}, \cite{R}), converges to this solution in the same setting.  Essential interest  is to find out whether a very weak solution has a subsequence which converges in the  sense of distributions, that is in the weak sense.  Usually, this is a weak solution to (\ref{jed}). 

Our approach is exposed in Section \ref{s 4} where the relations between sequences of infinitesimal generators and corresponding sequences of exponentially bounded integrated semigroups are discussed.  This is achieved by analysing associated sequences, the ones for which the norm of their difference tends to zero.  In this way, we introduce relations of equivalences in the corresponding spaces of moderate sequences given in Section \ref{s 2} and in Section \ref{s 4}.  A classical result on perturbation given in \cite{KH} fits well in the given approach.   

\subsection{Notation}
\label{ss not}

Let $(X,\|\cdot\|_X)$ be a Banach space and $(\mathcal{L}(X),\|\cdot\|_{\mathcal{L}(X)})$ be a space of linear continuous mappings on $X,$ with values in $X.$ For a sequence $ (x_n)_n\in X^{\mathbb N}$, $\mathbb N$ is the set of natural numbers, we say that it is moderate, and write $ (x_n)_n\in \mathcal{E}_X^M,$ if there exists $a\in\mathbb R$ such that $\|x_n\|_X=\mathcal{O}(n^{a}),$ which means $\|x_n\|_X\leq C n^a,$ $n>n_0$ for some $C>0.$ With $\mathcal{L}(X)$ instead of $X,$ we define $\mathcal{E}_{\mathcal{L}(X)}^M$. Note that $\mathcal{L}(X)$ is considered as a Banach algebra with respect to the operation of composition. Directly from the definition, one can deduce the following: $(R_n)_n\in \mathcal{E}_{\mathcal{L}(X)}^M$ if and only if there exists $(M_n)_n\in \mathcal{E}^M_{\mathbb{R}}$ such that $\|R_n x\|_X\leq M_n \|x\|_X,$ $x\in X$. 

Denote by $\mathcal{C}^M([0,\infty);X)$ the space of vector valued sequences  of  continuous mappings $F_n:[0,\infty)\ni t\mapsto F_n(t)\in X,$ $n\in\mathbb{N},$  with the property
\begin{equation}\label{SX_M} 
(\exists a\in\mathbb R)\quad \sup_{t\geq 0} \|F_n (t)\|_{X}=\mathcal{O}(n^{a}),\quad n\to \infty.
\end{equation} 
We will consider the  case when $X=L^p(\mathbb R^d),$ $p\in(1,\infty),$ but our special interest is the case  when we have above  $\mathcal L(X)$ instead of $X$. Then one obtains the space of sequences $(S_n)_n$ of strongly continuous mappings $S_n:[0,\infty)\to (\mathcal{L}(X),\|\cdot\|_{\mathcal{L}(X)}),$ $n\in\mathbb{N},$ denoted by $\mathcal{C}^M([0,\infty);\mathcal{L}(X)).$ The introduced sequences will be also denoted by $(S_n(t))_n,$ $t\geq 0,$ to emphasize the role of $t.$ Clearly, $\mathcal{C}^M([0,\infty);\mathcal{L}(X))$ is an algebra under composition.

The space of sequences  of continuous mappings $F_n:[0,\infty)\to (X,\|\cdot\|_{X}),$  with the property $\sup_{t\geq 0} \|e^{-\omega t} F_n (t)\|_{X}=\mathcal{O}(n^{a}),$ $n\to \infty, $  for some $ \omega>0$  and some $a\in\mathbb R,$ is denoted by $\mathcal{C}^M_{\rm exp}([0,\infty);X).$ It is also an algebra.  Again, we emphasize the case when $\mathcal{L}(X)$ is instead of $X,$ and write $(S_n)_n\in \mathcal{C}^M_{\rm exp}([0,\infty);\mathcal{L}(X))$ if
\begin{equation}\label{exp1}
\sup_{t\geq 0} \|e^{-\omega t} S_n (t)\|_{\mathcal{L}(X)}=\mathcal{O}(n^{a}),\; n\to \infty, \mbox{ for some } \omega>0 \mbox{ and some } a\in\mathbb R.
\end{equation} 
Note that for every $(S_n)_n\in \mathcal{C}^M_{\rm exp}([0,\infty);\mathcal{L}(X)) $ and every $t_0\in [0,\infty)$ we have  $(S_n(t_0))_n\in \mathcal{E}^M_{\mathcal{L}(X)}.$

Let us note that for all sequences under consideration we have to assume that their properties hold for $n>n_0$ since only the behaviour, as $n\rightarrow \infty,$ is important. In the sequel, we will not explicitly point out this fact and just assume that a certain property holds for every $n\in\mathbb{N}.$ Actually, it is not a restriction since we can always change the first $n_0$ elements by the $(n_0+1)$-th element.

\section{Generalized exponentially bounded integrated semigroups}
\label{s 2}

\subsection{Sequences of generators}
\label{ss sog}

Let $(A_n)_n$ be a sequence of closed linear operators acting on $X$ and $D_{A_n}$ be a domain for $A_n,$ $n\in\mathbb{N}.$  Let $(R(\lambda,A_n))_n$ be a sequence of resolvents that corresponds to  $(A_n)_n$ and $\rho(A_n),$ $n\in\mathbb{N},$ be their resolvent sets.  Assume:
\begin{enumerate}
\item[(G1):] There exists $D\subset X,$ $D\neq \emptyset,$ such that $D_{A_n}=D,$ $n\in\mathbb{N}.$
\item[(G2):] There exists $\omega>0$ such that $(\omega,\infty)\subset\rho(A_n),$ $n\in\mathbb{N}.$ 
\item[(G3):] $(R(\lambda,A_n))_n\in \mathcal{E}^M_{\mathcal{L}(X)},$ $\lambda\in (\omega,\infty).$
\end{enumerate}
Denote by $\mathcal{A}^{M}_3$ the set of sequences which satisfy (G1) - (G3).  We call $(A_n)_n$ a sequence of generators. 

We define a domain $\mathbf{D}_A$ for $(A_n)_n\in \mathcal{A}^{M}_3,$ that is $(A_n)_n:\mathbf{D}_A\subset \mathcal{E}_X^M\to \mathcal{E}_X^M,$ as
\begin{align*}
\mathbf{D}_A=\left\{(x_n)_n\in \mathcal{E}_X^M\,:\,x_n\in D,\,n\in\mathbb{N}\;\;\wedge\; (A_n x_n)_n\in \mathcal{E}_X^M\right\}.
\end{align*}

\begin{Proposition} \label{pr21}
Let $(A_n)_n$ be a sequence of generators and $(y_n)_n\in \mathcal{E}_X^M.$ Then $(R(\lambda,A_n)y_n)_n=(x_n)_n\in \mathbf{D}_A,$ $\lambda\in (\omega,\infty).$ Conversely, if $(x_n)_n\in \mathbf D_A$, then  there exists $(y_n)_n\in\mathcal E^M_X$ so that $R(\lambda, A_n)y_n=x_n, $ $n\in \mathbb N,$ for $\lambda\in(\omega,\infty).$
\end{Proposition}

\begin{proof}
Let $\lambda\in (\omega,\infty).$ It is clear that $x_n=R(\lambda,A_n)y_n\in D$ for every $n\in\mathbb{N}$ and that $(x_n)_n=(R(\lambda,A_n)y_n)_n\in \mathcal{E}_X^M,$ $\lambda\in (\omega,\infty).$  Finally, since $(x_n)_n,(y_n)_n\in \mathcal{E}_X^M$ we have $(A_n x_n)_n=(\lambda x_n - y_n)_n\in \mathcal{E}_X^M.$

For the converse assertion, just note that if $y_n=\lambda x_n-A_nx_n, $ $n\in\mathbb N$, then $(y_n)_n\in \mathcal{E}_X^M$ and $R(\lambda,A_n)y_n=x_n, $ $n\in\mathbb N.$
\end{proof}

A range $\mathbf{R}_{\lambda,A},$ $\lambda\in (\omega,\infty),$ of the sequence of resolvents $(R(\lambda,A_n))_n,$ that corresponds to the sequence of generators $(A_n)_n,$ is defined as:
\begin{align*}
\mathbf{R}_{\lambda,A}=\left\{(x_n)_n\in \mathcal{E}_X^M:\text{there exists } (y_n)_n\in \mathcal{E}_X^M \text{ so that }  (R(\lambda,A_n)y_n)_n=(x_n)_n\right\}.
\end{align*}
Since for every $\lambda,\lambda'\in (\omega,\infty)$ one obtains $\mathbf{R}_{\lambda,A}=\mathbf{R}_{\lambda',A},$ we will use notation $\mathbf{R}_A=\mathbf{R}_{\lambda,A},$ $\lambda\in (\omega,\infty).$ Now we state the direct consequence of Proposition \ref{pr21}.

\begin{Corollary}\label{cor D=R}
$\mathbf{R}_{A}=\mathbf D_A$.
\end{Corollary}

\subsection{Generalized exponentially bounded integrated semigroup and sequences of strong infinitesimal generators}
\label{ss geis}

\begin{definition}\label{def g.i.g.}
Let $(A_n)_n\in \mathcal{A}^{M}_3.$ It is called a sequence of infinitesimal generators if there exists a sequence   $(S_n)_n\in \mathcal{C}^M_{\rm exp}([0,\infty);\mathcal{L}(X)),$  that is (\ref{exp1}) holds, and
\begin{align}\label{rez}
R(\lambda,A_n)=\lambda\int_0^\infty e^{-\lambda t}S_n(t)\,dt,\quad \lambda\in (\omega,\infty),\quad n\in \mathbb{N};
\end{align}
$(S_n)_n$ is called a sequence of exponentially bounded integrated semigroups (in short, g.e.i.s.  or in plural g.e.i.s.'s) generated by $(A_n)_n.$
\end{definition}

\begin{Remark}
If \eqref{exp1} holds,  then by Theorem 3.1 of Arendt \cite{A 1987},  the necessary and sufficient condition  for $(S_n)_n$ to be a g.e.i.s. is that $R(\lambda,A_n), $ $\lambda\in(\omega,\infty),$ is a pseudoresolvent, for  every $n\in\mathbb N.$
\end{Remark}

A direct application  of Theorem 2.5.1 in \cite{ABHN 2011} gives in Theorem \ref{th 1} below the existence of a sequence of exponentially bounded integrated semigroups $S_n,$ $n\in\mathbb N,$ with the generators $A_n, $ $n\in\mathbb N,$ for which the conditions (G1) -  (G3) hold, as well as the next one,
\begin{align}\label{G5}
\sup_{{\rm Re}\,\lambda>\omega}\| \lambda^{b} R(\lambda,A_n)\|_{\mathcal{L}(X)} \leq M_n,\ n\in\mathbb{N},\;\text{ for some } \; b>0\text{ and } \; (M_n)_n\in \mathcal{E}^M_\mathbb{R}.
\end{align}

\begin{Theorem}\label{th 1}
Let $(A_n)_n \in \mathcal{A}^{M}_3$ so that it satisfies condition \eqref{G5}.Then there exists a sequence $(S_n)_n$ of exponentially bounded integrated semigroups such that
\begin{align*}
R(\lambda,A_n)=\lambda\int_0^\infty e^{-\lambda t}S_n(t)\,dt,\quad \lambda\in (\omega,\infty),\quad n\in\mathbb{N},
\end{align*}
and $(S_n)_n\in \mathcal{C}^M_{\rm exp}([0,\infty);\mathcal{L}(X)).$ More precisely, the growth condition for $(S_n)_n$ is given by
\begin{align*}
\sup_{t>0}\|e^{-\omega t}t^{-b}S_n(t)x\|_X\leq  M'_n,\;n\in\mathbb{N},\,\text{ for some } ( M'_n)_n\in \mathcal{E}^M_\mathbb{R}.
\end{align*}
\end{Theorem}

\begin{proof}
Let $x\in X$.  Then  assumption \eqref{G5} implies that for every $n\in\mathbb N,$
\begin{align*}
S_n(t)x=\frac{1}{2\pi}\int_{-\infty}^\infty e^{(\alpha+ir)t} \frac{R(\alpha+ir,A_n )x}{\alpha+ir}\,dr,\quad t\geq 0,
\end{align*}
where $\alpha>\omega$ and $S_n (\cdot)x\in C([0,\infty),X)$ (the space of continuous functions $[0,\infty)\rightarrow X$).  Moreover,  
\begin{align*}
\sup_{t>0}\|e^{-\omega t}t^{-b}S_n(t)x\|_X\leq M'_n,\quad n\in\mathbb{N},
\end{align*}
for some $(M'_n)_n\in \mathcal{E}^M_\mathbb{R}.$ This is a direct consequence of Theorems 3.2.8 and 2.5.1 in \cite{ABHN 2011} where, with fixed $n,$ we have to put $q(\lambda)=\lambda^{b-1}R(\lambda,A_n)x$ and use this theorems for $f(\cdot)=S_n(\cdot)x.$ Namely, as in \cite{ABHN 2011} at the very end of the proof of Theorem 2.5.1, with $R>0$, one has, 
\begin{equation*}
\|S_n(t)\|_{\mathcal{L}(X)}\leq \frac{M_n e^{\alpha t}}{\pi b R^b}+\frac{M_n e^{\alpha t} }{\pi R^b}\int_0^{\pi/2} e^{Rt\cos\theta} \,d\theta,\quad t>0,
\end{equation*}
where
\begin{equation*}
M_n=\sup_{{\rm Re}\,\lambda>\omega}\|\lambda^{b} R(\lambda,A_n)\|_{\mathcal{L}(X)}, \quad n\in\mathbb{N}.
\end{equation*}
Now, taking $R=1/t$ one obtains $\|e^{-\omega t}t^{-b}S_n(t)\|_{\mathcal{L}(X)}\leq CM_n=M'_n,$ $n\in\mathbb{N},$ $t>0.$ Clearly, $\sup_{t>0}\|e^{-\omega t}t^{-b}S_n(t)\|_{\mathcal{L}(X)}\leq M'_n$ which, for $\omega_1\geq \omega+b,$ implies
\begin{equation*}
\sup_{t\geq 0} \|e^{-\omega_1 t} S_n (t)\|_{\mathcal{L}(X)}\leq M'_n,\quad n\in\mathbb{N};
\end{equation*}
so, $(S_n)_n\in \mathcal{C}^M_{\rm exp}([0,\infty);\mathcal{L}(X)).$
\end{proof}

The obtained  growth condition for $(S_n)_n$ is stronger than the one which characterizes the growth in $\mathcal{C}^M_{\rm exp}([0,\infty);\mathcal{L}(X))$ because it gives the behaviour of the sequence $(S_n)_n$ when $t\to 0.$

\section{Revisiting of known examples }\label{examples}

All   one time integrated semigroups in this section are  well known (for fixed $n$).  They are used for the explanation of our approach to sequences of such semigroups.  Our main literature are results for integrated semigroups given in \cite{ABHN 2011},  Section 8.3. Concerning notation, if $t\mapsto f(t,x)$ is a continuous function on $[0,\infty)$ with values in the Schwartz space of distributions $\mathcal{D}'(\mathbb{R}^d),$ we write  $f(t,x)\in C([0,\infty),\mathcal D'(\mathbb R^d)).$ Additionally, if the above function  is continuously differentiable, we write $f(t,x)\in C^1([0,\infty),\mathcal D'(\mathbb R^d))$. These functions are elements of $\mathcal D'((0,\infty)\times\mathbb R^d)$ through the dual pairing $\langle f(t,x),\psi(t,x)\rangle,$ $\psi\in\mathcal{D}((0,\infty)\times\mathbb{R}^d),$ where $\mathcal{D}((0,\infty)\times\mathbb{R}^d)$ is a space of smooth function $\psi$ supported by a compact set in $(0,\infty)\times\mathbb{R}^d$ with the usual convergence structure. 

Recall \cite{ABHN 2011},  a smooth function $a$ on $\mathbb{R}^d$ is called a symbol belonging to $S^{m}_{1,0},$ $m\in\mathbb{N},$ if $|D^\alpha_\xi a(\xi)|\leq C \langle\xi\rangle^{m-|\alpha|},$ $ \xi\in \mathbb{R}^d,$  for some $C>0$ and all $\alpha\in (\mathbb N\cup\{0\})^d, $ where $ \langle\xi\rangle=(1+|\xi|^2)^{1/2}.$ Then $(a_n)_n\in (S^{m}_{1,0})^\mathbb{N}$ is a moderate sequence of symbols if there exists $(C_n)_n\in \mathcal{E}^M_\mathbb{R}$ so that  
\begin{equation}\label{dot}
|D^\alpha_\xi a_n(\xi)|\leq C_n \langle\xi\rangle^{m-|\alpha|},\quad \xi\in \mathbb{R}^d,\ n\in\mathbb{N}.
\end{equation}

With the notation $D=(D_1,...,D_d),$ $D_j=\partial/(i\partial x),$ $j=1,...,d,$ and $\mathcal{F}$ and $\mathcal{F}^{-1}$ for the Fourier and inverse Fourier transform, we consider pseudo-differential operator  formally defined by $a(D)f=({\rm Op}\  a)f= \mathcal{F}^{-1}(a \mathcal{F}f),$ where $a\in S^{m}_{1,0}$ and $f$ belongs to an appropriate space of functions or distributions. (Here, the notation $D$ for the differential operator should not be confused with $D=D_A\subset X,$ which is the domain of the corresponding operator $A.$) Usually, a sequence of  such operators $(a_n(D))_n$ can be considered as a stationary one $a_n=a,$ $ n\in\mathbb N,$  or as a sequence of approximations of $a$.

\begin{Remark}\label{rim}
The regularization  of a  Cauchy problem (\ref{jed}) with $u_0\in\mathcal D'(\mathbb R^d)$ and $f\in C([0,\infty),\mathcal D'(\mathbb R^d))$ leads to a family of Cauchy problems with $u_{0,n}$ and $f_n, $ $n\in\mathbb N,$ belonging to appropriate function spaces,
\begin{equation}\label{eq **}
\frac{\partial}{\partial t}w_n(t,x) - a_n(D)w_n(t,x) =f_n(t,x),\quad w_n (0,x)=u_{0,n}(x),\quad n\in\mathbb{N},
\end{equation}
as follows. Let $\theta\in\mathcal D(\mathbb R^d).$ Assume that it is non-negative and $\int_{\mathbb R^d}\theta(x)dx=1$. Let $\theta_n(x)=n^d\theta(nx),$ $ x\in\mathbb R^d$; this is a delta sequence. In the case when $u_0(x)$ is a distribution and  $f(t,x)\in C([0,\infty),\mathcal D'(\mathbb R^d))$, we make  regularization by the use of convolution: 
\begin{align*}
u_{0,n}(x)=u_0(x)*\theta_n(x), \quad f_n(t,x)= f(t,x)*_x\theta_n(x),\quad n\in \mathbb N,\  t\geq 0, \ x\in\mathbb R^d.
\end{align*}
\end{Remark}

In order to show that the regularizations of Remark \ref{rim} determine elements of the  domain $\mathbf{D}_A$ (cf. Subsection \ref{ss sog})  related to the pseudo-differential operators $A_n=a_n(D)=({\rm Op}\;a_n),$ $n\in\mathbb N,$  we recall that $g\in\mathcal D'_{L^p}(\mathbb R^{d}),$ $ p\in(1,\infty],$ if and only if  it is of the form $g=\sum_{|\alpha|\leq k} g_\alpha^{(\alpha)}$, where $g_\alpha\in L^p(\mathbb R^d),$ $k\in\mathbb{N}\cup\{0\}.$  Recall also, that $\mathcal D'_{L^p}(\mathbb R^{d})$ is a strong dual of the space  $\mathcal D_{L^q}(\mathbb R^{d}),$ $ q=p/(p-1)$ ($q=1$ for $p=\infty$), consisting of smooth functions $\phi$ for which all the norms $\|\phi^{(\alpha)}\|_{L^q(\mathbb R^d)},$ $ \alpha\in (\mathbb N\cup\{0\})^d,$ are finite. Using the H\" older inequality $\|g_\alpha*\theta_n^{(\alpha)}\|_{L^p}\leq \|g_{\alpha}\|_{L^p}\|\theta_n^{(\alpha)}\|_{L^q}$, we have 
\begin{equation}\label{lpp}
g_n=\sum_{|\alpha|\leq k} g_\alpha*\theta_n^{(\alpha)}\in L^p(\mathbb R^d), \;\text{ since }\; g_\alpha*\theta_n^{(\alpha)}\in L^p(\mathbb R^d),\ n\in\mathbb N.
\end{equation}  
Finally,  $(\theta_n^{(\alpha)}(x))_n=(n^{d+|\alpha|}\theta^{(\alpha)}(nx))_n,$ $x\in\mathbb{R}^d,$ imply $(\theta^{(\alpha)}_n)_n\in \mathcal E^M_{L^q(\mathbb R^{d})}.$ So, $(g_n)_n\in\mathcal{E}^M_{L^p(\mathbb{R}^{d})}.$ 

\begin{Remark} The case $p=1$ should be treated in another way. We exclude this case in order to simplify our exposition.
\end{Remark}

Recall, if $a\in S^m_{1,0},$ then it determines a pseudo-differential operator on $L^p(\mathbb R^d)$,  with the domain 
\begin{align*}
D_{{\rm Op}\,a}=\{g\in L^p(\mathbb R^d): \mathcal F^{-1}(a(\xi)\mathcal Fg(\xi))\in L^p(\mathbb R^d)\}.
\end{align*}
Since $\mathcal S(\mathbb R^d)\subset D_{{\rm Op}\,a},$  these operators are densely defined.  Moreover, we have the next lemma.

\begin{Lemma}
Let $p\in[1,\infty),$ and $g_n,$ $ n\in \mathbb N,$ be of the form (\ref{lpp}). Let $(a_n)_n\in (S^m_{1,0})^{\mathbb N}$ so that (\ref{dot}) holds. Then,  $(({\rm Op}\  a_n)g_n)_n$ belongs to $ \mathcal{E}^M_{L^p(\mathbb{R}^{d})}.$
\end{Lemma}

\begin{proof}
Take $s\in\mathbb N$  such that $h_n(x),$ $x\in\mathbb R^d$, defined by $x\mapsto h_n(x)=\int_{\mathbb R^d}  e^{2\pi i \xi x}a_n(\xi)(1+2\pi |\xi|^2)^{-s} d\xi$ belongs to $L^q(\mathbb R^d)$. Then
\begin{align}\label{eq :}
({\rm Op}\  a_n)g_n(x) & =\int_{\mathbb{R}^n}  (1- \Delta_x)^s e^{2\pi i \xi x}\frac{a_n(\xi)}{(1+2\pi |\xi|^2)^s}\ d\xi \ast g_n(x)\nonumber\\
& = h_n(x) \ast (1- \Delta_x)^s g_n(x) \qquad(\Delta_x \text{ is Laplacian}). 
\end{align}
By (\ref{dot}), one can find $C>0$ and $a\in\mathbb{R}$ so that $\|h_n\|_{L^q}\leq Cn^a,$ $ n\in\mathbb N. $ We use  (\ref{lpp}) and in $(1-\Delta_x)^s g_n(x),$ on the right hand side of (\ref{lpp}), we differentiate only the part   $\theta_n^{(\alpha)}$. Clearly, $(\theta^{(\alpha)}_n)_n\in \mathcal E^M_{L^q(\mathbb R^{d})}.$ So, using again H\" older inequality, one obtains that there exists a sequence $(C_n)_n\in\mathcal E^M_{\mathbb R}$ so that  
\begin{align*}
\|{(\rm Op}\  a_n) g_n\|_{L^p(\mathbb{R}^d)}\leq C_n,\quad n\in\mathbb N. 
\end{align*}
This completes the proof.
\end{proof}

We continue with the assumptions (cf. \cite{ABHN 2011} Subsections 8.2, 8.3): \begin{enumerate}
\item[(A1):] 
$\hspace{2,75cm}\exists\, r>0, \ \exists\,  L>0, \ \exists\, C_n>0\; \;n\in\mathbb N, \ \exists\, c_0>0,
$
$$ |a_n(\xi)|\geq C_n |\xi|^r,  \ |\xi|>L \mbox{ and } 1/C_n\leq c_0;$$
\item[(A2):] $\rho({\rm Op}\ a_n)\neq \emptyset,$ $n\in\mathbb N;$ 
\item[(A3):] $\sup_{\xi\in\mathbb{R}^d}{\rm Re}\ a_n(\xi)\leq m,$ $n\in\mathbb{N},$  for some $m\in \mathbb{R}.$
\end{enumerate}

\begin{Proposition} \label{pret} Assume that a sequence of symbols $(a_n)_n,$  satisfies \eqref{dot}, \eqref{delc},  (A1) - (A3), as well as  that all ${\rm Op}\ a_n$ have the same domain $D=D_{{\rm Op}\,a_n}, $ $n\in\mathbb N.$ 
Assume that $p$ satisfies
\begin{equation}\label{delc}
\left|\frac{1}{2}-\frac{1}{p}\right|<\frac{r}{md}. 
\end{equation}
Then
\begin{equation}\label{HKe}
S_n(t)u=\mathcal F^{-1}\left(\int_0^t e^{sa_n(\cdot)}ds\ \mathcal F u(\cdot)\right),\quad u\in L^p(\mathbb R^d),\quad n\in\mathbb N,
\end{equation}
is a g.e.i.s. generated by $({\rm Op}\ a_n)_n.$
Moreover,  $({\rm Op}\ a_n)_n\in\mathcal{A}^{M}_3.$ In particular,  (G3) holds with $\sup_{n\in\mathbb N}\|R(\lambda,{\rm Op}\  a_n)\|_{\mathcal L(L^p)}<\infty, $ $\lambda\in(\omega,\infty). $
\end{Proposition}

\begin{proof}
Essentially, assumption (\ref{delc}) implies that $a_n$ determines one time integrated semigroup (cf. \cite{ABHN 2011}). Moreover, by  the implication (i)  $\Rightarrow$ (ii) of Theorem 8.3.6 in \cite{ABHN 2011}  we have directly that $a_n$ determines exponentially bounded integrated semigroup $S_n$ of the form (\ref{HKe}) for  every $n$. By the uniform bound of $1/C_n$ in (A1) and the uniform bound   in (A3), we obtain that $(S_n)_n$ is g.e.i.s.  Since $R(\lambda,{\rm Op}\ a_n)$ is defined by the  $S_n$  with the uniform exponential bound  of all $S_n$ in (\ref{rez}), $n\in \mathbb N,$ we have that (G2) holds with $\omega > |m|,$ as well as that (G3) holds with the uniform bound $\sup_{n\in\mathbb N,\lambda>\omega}\|R(\lambda,{\rm Op}\ a_n)\|_{\mathcal{L}(L^p)}<\infty.$
\end{proof}

\begin{Proposition}\label{3}
Concerning $(a_n)_n$ assume that all the assumptions of Proposition \ref{pret} hold. Let $(u_{0,n})_n\in \mathcal{E}^M_{L^p(\mathbb{R}^d)}$ and  $(f_n)_n, (\frac{d}{dt} \,f_n)_n\in \mathcal{C}^M_{exp}([0,\infty);L^p(\mathbb{R}^{d})),$ $p\in [1,\infty).$  Then the sequence of equation 
\begin{align}\label{equiv1}
w_n (t,x)=u_{0,n}(x)+a_n(D)\int_0^t w_n (r,x)\ dr +\int_0^t  f_n(s,x)\ ds,\quad n\in\mathbb{N},
\end{align}
has a sequence of solutions  $(w_n)_n\in \mathcal{C}^M_{\rm exp}([0,\infty);L^p(\mathbb R^d))$ (mild solutions, for every $n$)  given by $w_n=\frac{d}{dt} v_n$, where
\begin{align}\label{sol11}
v_n(t,x)=S_n(t)u_{0,n}(x) + \int_0^t S_n(t-r)f_n(r,x)\ dr,\quad t\geq 0,\quad x\in \mathbb R^d,\quad n\in\mathbb{N}.
\end{align}
\end{Proposition}

\begin{proof} We have by Proposition \ref{pret} that $(S_n)_n$ is a g.i.e.s.  so the mappings   $[0,\infty)\rightarrow L^p(\mathbb R^d)$ given by $t\mapsto S_n(t)u_{0,n}(x)$ and $t\mapsto \int_0^t S_n(t-r)f_n(r,x)\ dr$ are continuous, as well as  their derivatives, with respect to $t$, for every fixed $n$. Thus, by assumptions of the proposition, there holds that $a_n(D)u_{0,n}+\frac{d}{dt}f(0)\in L^p(\mathbb R^d), n\in\mathbb N.$ This implies that the assumptions  of  Corollary 3.2.11. in \cite{ABHN 2011} are satisfied. By part c)  of this corollary,  there exists a unique mild solution to (\ref{equiv1}) (with fixed $n$). The fact that $(v_n)_n$ and $(w_n)_n$ have a moderate growth with respect to $n$  follows from the moderate growth of $(S_n)_n,$ $(u_{0,n})_n$ and $(f_n)_n$.
\end{proof}

The sequence of mild solutions $(w_n)_n$ is a very weak solution to (\ref{jed}), in the sense of \cite{GR} and \cite{R} because for every fixed $n$, $w_n(\cdot,\cdot)$ is the distributional solution to (\ref{eq **}),
\begin{equation}\label{dise}
\langle \frac{\partial}{\partial t}w_n(t,x)-a_n(D)w_n(t,x)-f_n(t,x),\psi(t,x)\rangle=0,\quad \psi(t,x)\in\mathcal D([0,\infty)\times\mathbb R^d),
\end{equation}
$w_n(0,x)=u_{0,n}(x), n\in\mathbb N,$ and $(w_n)_n$ has a moderate growth with respect to $n$. Moderate growth means that 
\begin{equation}\label{mod}
\forall \psi\in\mathcal D((0,\infty)\times\mathbb R^d)\;\; \exists m=m_\psi\in\mathbb R,
\end{equation}
\begin{equation*}
|\langle w_n(t,x),\psi(t,x)\rangle|=O(n^m), \;n\rightarrow \infty.
\end{equation*} 
This is a consequence of the fact that the  mapping  $[0,\infty)\times L^p(\mathbb R^d) \ni (t,f) \mapsto S_n(t)f\in L^p(\mathbb R^d)$ is continuous because  $t\mapsto S_n(t,\cdot)\in  L^p(\mathbb{R}^d)$ is continuous,  and determines a distribution $w_n(t,x)\in\mathcal D'((0,\infty)\times\mathbb R^d)$. Note that $C([0,\infty),L^p(\mathbb{R}^d))=C([0,\infty)\times L^q(\mathbb{R}^d)),$ with  $f(t,\varphi)=\int f(t,x)\varphi(x)\,dx,$ $\varphi\in L^q(\mathbb{R}^d).$

The next corollary serves as a motivation for our approach. 

\begin{Corollary} \label{34}
 Let  $a(D)\in S^m_{0,1}$ be a pseudo-differential operator on $L^p(\mathbb R^d)$ so  that  (A1) - (A3) and  (\ref{delc}) hold. Let $(u_{0,n}(x))_n$ and $(f_n(t,x))_n$ be a sequence in $L^p(\mathbb R^d)$ and $C^1([0,\infty),L^p(\mathbb R^d))$, respectively, obtained as a regularization of $u_0\in\mathcal D'_{L^p}(\mathbb R^d)$ and $f\in C^1([0,\infty),\mathcal D'_{L^p}(\mathbb  R^d))$  (as in Remark \ref{rim}).  
Then, the sequence $(v_n)_n$ of the form (\ref{sol11})  determine  $w_n=\frac{d}{dt}v_n, n\in\mathbb N$, a sequence of mild solutions to (\ref{equiv1}); $(w_n)_n$ has a subsequence $(w_{k_n})_n$ with elements in $C([0,\infty),L^p(\mathbb R^d))$  such that it converges to $w(t,x)\in \mathcal D'([0,\infty)\times\mathbb R^d)$. Moreover, $w$ is a weak solution to (\ref{jed}); it satisfies
\begin{align}\label{zad1}
\langle \frac{\partial}{\partial t}w(t,x)-a(D)w(t,x)-f(t,x),\psi(t,x)\rangle=0,\quad \psi(t,x)\in\mathcal D([0,\infty)\times\mathbb R^d),
\end{align}
\begin{equation*}
\langle w(0,x),\psi(x)\rangle=\langle u_0(x),\psi(x) \rangle, \; \psi\in\mathcal D(\mathbb R^d). 
\end{equation*}
\end{Corollary}

\begin{proof}
We fix $t>0.$ With $S$ as in (\ref{HKe}) (without subindex), we have that $(S(t)u_{0,n}(\cdot))_n$ is a bounded sequence in $\mathcal D'_{L^p}(\mathbb R^d)$. The same holds for $(\int_0^t S(t-r)f_n(r,\cdot)\ dr)_n.$   This implies that there exists a subsequence $v_{k_n}(t,x)=S(t)u_{0,k_n}(x) + \int_0^t S(t-r)f_{k_n}(r,x)\ dr$ so that it converges weakly as $n\rightarrow\infty$ to $v(t,\cdot)\in\mathcal D'_{L^p}(\mathbb R^d)$. If we consider the set of rational points $Q_+\subset[0,\infty),$ $Q_+=\{q_1,,q_2,...\}$ and form a convergent subsequence of already convergent subsequence, by diagonalization, we can construct a subsequence (again denoted as) $(v_{k_n})_n$ so that for every $t\in Q_+,$ $v_{k_n}(t,\cdot)\rightarrow v(t,\cdot)\in \mathcal D'_{L^p}(\mathbb R^d),$ $n\to\infty.$ Since all the elements of this subsequence are continuous with respect to $t$, we obtain that  $v_{k_n}(t,\cdot)\rightarrow v(t,\cdot), t\in[0,\infty), n\rightarrow \infty,$ where $v(t,\cdot)\in C^1([0,\infty),\mathcal D'_{L^p})\subset \mathcal D'([0,\infty)\times\mathbb R^d).$ This is a consequence of the fact that $t\mapsto \langle v_{k_n}(t,x),\psi(x)\rangle,$ $n\in\mathbb{N},$ $\psi\in\mathcal{D}_{L^q}(\mathbb{R}^d)$ as well as  $t\mapsto \langle \frac{d}{dt}v_{k_n}(t,x),\psi(x)\rangle,$ $n\in\mathbb{N},$ $\psi\in\mathcal{D}_{L^q}(\mathbb{R}^d)$ are uniformly continuous sequences of functions on any bounded interval $[0,T],$ $T>0.$ Thus, by the convergence in this space of distributions, $w=\frac{d}{dt}v$ is a weak solution to (\ref{jed}), i.e. \eqref{zad1} holds.
\end{proof}

Assume that $u_0\in L^p(\mathbb R^d),$  $f\in C^1([0,\infty),L^p(\mathbb R^d))$, $|D^\alpha_\xi a_n(\xi)|\leq C \langle\xi\rangle^{m-|\alpha|}, $ $n\in\mathbb N, $ $\xi\in \mathbb{R}^d, $ for some $C>0,$ as well as 
that (\ref{delc}) and (A1) - (A3) hold. Then,  the sequence of equations (\ref{equiv1}) (with $u_0$ and $f$ instead of $u_{0,n}$ and $f_n$) has a sequence of solutions $(w_n)_n$ of the form  (\ref{sol11}), where $S_n(t)$ is given by (\ref{HKe}). Moreover, there exists a subsequence of solutions $(w_{k_n})_n$ such that $w_{k_n}\rightarrow w, $ $n\rightarrow \infty,$  weakly in $ \mathcal D'([0,\infty)\times \mathbb R^d).$ 
This can be proved by the similar arguments as in the proof of the previous corollary.

We apply the above considerations to a special equation (\ref{equiv1}), in the case $d=1$ in order to discuss its dependence on the sequences of coefficients:  
\begin{align}\label{pex}
\frac{\partial}{\partial t} w_n- P_n\left(\frac{\partial}{\partial x}\right)w_n=f_n, \quad n\in\mathbb{N}, 
\end{align}	
where  $(f_n)_n$ is a moderate sequence in $C^1([0,\infty),L^p(\mathbb{R}))$ and $P_n$ is a linear  differential  operator with constant coefficients belonging to $\mathcal E^M_{\mathbb C}$, of the form
\begin{align*} P_n(\partial/\partial x)=\alpha_{0,n}+i\beta_{0,n}+(\alpha_{1,n}+i\beta_{1,n})\partial/\partial x+(\alpha_{2,n}+i\beta_{2,n})(\partial/\partial x)^2,\quad n\in\mathbb{N}.
\end{align*}
We note that (G1) holds since all the domains $D_{P_n}, n\in\mathbb N$, are equal to the Sobolev space $W^{2,p}(\mathbb R)$. Since we are considering one dimensional case,  $P_n$ are elliptic, $n\in\mathbb N,$ so (A1) and (A2)  are fulfilled. A sufficient condition for the application of Proposition \ref{3} and Corollary \ref{34}, originating from  (A3) reads: 
\begin{align*}
\alpha_{2,n}\geq 0 \text{ and } \omega_n=\max\left\{0,\frac{4\alpha_{2,n}\alpha_{0,n}+\beta_{1,n}^2}{4\alpha_{2,n}}\right\}\leq \omega,\ n\in\mathbb{N},  \text{ for some }  \omega\in \mathbb{R}.
\end{align*} 
It shows that whenever $\beta_{1,n}=O(\sqrt{\alpha_{2,n}}), $ $\alpha_{2,n}>0$ and $\alpha_{0,n}=O(1)$ condition (A3) is satisfied. Then, directly using \cite{KH} Theorem 4.1, one has that g.e.i.s. $(S_n)_n$ is defined by $S_n (t)u(t,\cdot)=(2\pi)^{-1/2}(\mathcal{F}^{-1}\phi_{t,n})\ast u(t,\cdot),$ $n\in\mathbb{N}$ where $\phi_{t,n}(\xi):=\int_{0}^t e^{p_n(i\xi)s}ds,$ $\xi\in \mathbb{R},$ $ t>0$, see (\ref{HKe}).  Again one has the existence of a very weak solution for  (\ref{pex}).

Note that by  \cite{KH}, instead of $L^p(\mathbb{R})$ one can consider in (\ref{pex}) spaces: $ C_0(\mathbb{R}),$ $C_b(\mathbb{R}),$ $UC_b(\mathbb{R}).$ More generally, if one defines the space $\mathcal D_{E}$, where $E$ is one of quoted spaces, then by the adaptation of Proposition \ref{3} and Corollary \ref{34}, one obtains the corresponding g.e.i.s's. We refer to \cite{DPV} for the translation invariant spaces $\mathcal D_E$ and their duals. 
 
\begin{Remark}
We  can consider equation (\ref{eq **}) with
$a_n(\xi)=ic_n |\xi|^m,$ $\xi\in \mathbb{R}^d,$ $m\in \mathbb{R},$ $c_n\in \mathbb{R},$ $ |c_n|\leq c,$ $n\in\mathbb{N},$ with the similar definitions:  $({\rm Op}\ a_n)(u)=\mathcal F^{-1}(a_n\mathcal F u)$ and their  domains $D_n=D\subset L^p(\mathbb{R}^d),$ $n\in\mathbb{N}.$  Now as in \cite{ABHN 2011}, Example 8.2.5, let $m$ and $p$ satisfy conditions of Theorem 8.3.9 of \cite{ABHN 2011} (with $k=1$). Then the sequence  $(a_n)_n\in \mathcal{A}^{M}_3$ and it determines a g.e.i.s.  $(S_n)_n$.  We have a similar assertions as in Proposition  \ref{3} and Corollary \ref{34}, adapted to $(ic_n|\xi|^m)_n$, which will not be repeated. 
\end{Remark}                

\section{Associated sequences }
\label{s 4}

In this section we classify  infinitesimal generators and corresponding g.e.i.s.'s and we analyse the  relations of generalized infinitesimal generators and g.e.i.s.'s. Moreover, we introduce sequences associated to zero within algebras of Subsection \ref{ss not}.

The notion of association between  sequences is well understood in the literature related to the algebraic theory of  generalized function   \cite{C1984}, \cite{MMMR}. A sequence $(x_n)_n\in\mathcal E^M_X$ is associated to zero if $||x_n||_X \rightarrow 0$ as $n\rightarrow \infty$. Denote by $\mathcal{I}_{X}$ the space of elements of $\mathcal E^M_X$  which are associated  to zero. Such elements make a  subspace  of $\mathcal E^M_X$.  Similarly,  we define $\mathcal{I}_{\mathcal L(X)}$ as the space of sequences $(N_n)_n\in\mathcal E^M_{\mathcal L(X)}$ which converges to zero in $\mathcal L(X),$ as  $n\rightarrow \infty$; $\mathcal{I}_{\mathcal L(X)}$  is a subalgebra of $\mathcal E^M_{\mathcal L(X)}$ under the operation of composition.  A subspace of $\mathcal{C}^M_{\rm exp}([0,\infty);X),$  consisting of elements $(N_n(t))_n,$ $t\geq 0,$  with the property 
\begin{align*}
\sup_{t\geq 0} \|e^{-\omega t} N_n (t)\|_{X}\rightarrow 0,\quad n\to \infty, \mbox{ for some } \omega>0,
\end{align*}
is denoted by $\mathcal{I}_{\rm exp}([0,\infty);X).$  Analogously, we define a subspace $\mathcal{I}_{\rm exp}([0,\infty);\mathcal{L}(X))$ of the space $\mathcal{C}^M_{\rm exp}([0,\infty);\mathcal{L}(X))$ containing elements $(N_n(t))_n,$ $t\geq 0,$  such that for some $ \omega>0,$ $\sup_{t\geq 0} \|e^{-\omega t} N_n (t)\|_{\mathcal{L}(X)}\rightarrow 0,$ $n\to \infty.$ 

Two sequences in $\mathcal E^M_X$ or  $\mathcal E^M_{\mathcal L(X)}$ or $\mathcal{C}^M_{\rm exp}([0,\infty);X)$ or $\mathcal{C}^M_{\rm exp}([0,\infty);\mathcal{L}(X))$ are associated in these spaces, respectively,   if their difference converges to zero,  that is, belongs to the corresponding space $\mathcal{I}_{X}$ or  $\mathcal{I}_{\mathcal L(X)}$ or $\mathcal{I}_{\rm exp}([0,\infty);X)$ or $\mathcal{I}_{\rm exp}([0,\infty);\mathcal{L}(X)).$  In any of these spaces the association is the relation of equivalence. We will use in the sequel the symbol ''$\sim$'' if the difference of two elements is associated. 
 
\begin{Remark}
One can also define weak associative sequences of quoted algebras when one involve test functions in the definitions (cf. \cite{C1984}), as is it implicitly suggested in Section \ref{examples}, where we considered sequences with values in distribution spaces. 
\end{Remark}
 
Concerning generators, we add to conditions (G1),  (G2) and (G3) the following one:
\begin{enumerate}
\item[(G4):]  For every $\lambda>\omega$ there exist $0<c_1(\lambda)<c_2(\lambda)$ such that
\begin{align*}
\|R(\lambda,A_n)\|_{\mathcal L(X)}\in(c_1(\lambda),c_2(\lambda)), \quad n\in\mathbb{N}. 
\end{align*}
\end{enumerate}
Note that (G4) implies 
\begin{align*}
R(\lambda,A_n)y_n\rightarrow 0 \;\mbox{ if and only if } \; y_n\rightarrow 0, \; n\rightarrow \infty\; \text{for all}\; \lambda\in (\omega,\infty).
\end{align*}
We denote by $\mathcal{A}^{M}_4$, the set of sequences which satisfy (G1), (G2), (G3), (G4).  

\begin{Lemma}\label{N u N}
Let $A\in \mathcal{A}^{M}_4.$ If $(x_n)_n\in \mathbf{D}_A$ and $(x_n)_n\in \mathcal{I}_{X}$ then $(A_n x_n)_n\in \mathcal{I}_{X}.$
\end{Lemma}

\begin{proof}
Denote by $y_n=\lambda x_n-A_n x_n,$ $n\in\mathbb{N}, $ $\lambda>\omega.$ Then  $(x_n)_n\in \mathbf{D}_A$ imply $(y_n)_n\in \mathcal{E}_X^M.$ Since $(R(\lambda,A_n)y_n)_n=(x_n)_n \in \mathcal{I}_X,$ according to (G4), one obtains $(y_n)_n\in \mathcal{I}_X.$ Finally, $A_n x_n=\lambda x_n - y_n, $ $\lambda>\omega,$ $n\in\mathbb{N}, $ implies $(A_n x_n)_n\in \mathcal{I}_X.$
\end{proof}

We introduce the relation of equivalence in $\mathcal{A}^{M}_4$: $(A_n)_n\simeq (\tilde A_n)_n$ if
\begin{enumerate}
\item[(GE1):] $D=\tilde D,$ where $D=D_{A_n}$ and $\tilde D=D_{\tilde A_n},$ $n\in\mathbb{N};$
\item[(GE2):] $\mathbf{D}_A=\mathbf{D}_{\tilde{A}};$
\item[(GE3):] $((A_n -\tilde A_n)x_n)_n\rightarrow 0$ in $X,$ as $n\to\infty,$ for every $(x_n)_n\in \mathbf{D}_A.$ 
\end{enumerate}

\noindent Note, if $(A_n)_n,$ $(\tilde A_n)_n $ in $\mathcal{A}^{M}_4$ then there always exists $\omega>0$ such that $(\omega,\infty)\subset \rho(A_n)\cap\rho(\tilde A_n),$ $n\in\mathbb{N}.$

We say that sequences of resolvents $(R(\lambda,A_n))_n$ and $(R(\lambda,\tilde A_n))_n, $ $\lambda>\omega,$ are associated,  and write 
\begin{align*}
(R(\lambda,A_n))_n \simeq  (R(\lambda,\tilde A_n))_n
\end{align*}
if for every $\lambda\in (\omega,\infty),$ 
\begin{enumerate}
\item[(RE1):] $R={\rm range\,}R(\lambda,A_n)={\rm range\,}R(\lambda,\tilde A_n),$ $n\in\mathbb{N};$ 
\item[(RE2):] $\mathbf{R}_{A}=\mathbf{R}_{\tilde A};$ 
\item[(RE3):] $((R(\lambda,A_n)-R(\lambda,\tilde A_n))y_n)_n\rightarrow 0$ in $X,$ as $n\to\infty,$ for every $(y_n)_n\in \mathcal{E}_X^M.$ 
\end{enumerate}

\begin{Theorem}\label{prop 1}
Let $(A_n)_n,(\tilde A_n)_n\in \mathcal{A}^{M}_4$ and  $(R(\lambda,A_n))_n, (R(\lambda,\tilde A_n))_n\in \mathcal{E}^M_{\mathcal{L}(X)}, $ $\lambda>\omega,$ be corresponding resolvents.  

Then $(A_n)_n \simeq(\tilde A_n)_n$ if and only if $(R(\lambda,A_n))_n  \simeq (R(\lambda,\tilde A_n))_n,$ $\lambda\in (\omega,\infty).$
\end{Theorem}

\begin{proof} 
\begin{enumerate}
\item[$(\Rightarrow)$] Let us first assume that $(A_n)_n\simeq (\tilde A_n)_n.$ Since $D_{A_n}=D_{\tilde A_n}=D,\,$ $n\in\mathbb{N},$ one directly obtains that ${\rm range\;}R(\lambda,A_n)={\rm range\;}R(\lambda,\tilde A_n)=D$ for any $\lambda\in (\omega,\infty)$ and $n\in\mathbb{N}.$ Also,  Corollary \ref{cor D=R} shows that $\mathbf{D}_A=\mathbf{D}_{\tilde A}$ implies $\mathbf{R}_A=\mathbf{R}_{\tilde A}.$

Let  $(y_n)_n \in \mathcal{E}_X^M$ arbitrary and $\lambda>\omega.$ Denote by
\begin{align*}
(R(\lambda,A_n)y_n)_n=(x_n)_n, \; (R(\lambda,\tilde A_n)y_n)_n=(\tilde x_n)_n\in \mathbf{R}_A=\mathbf{D}_A.
\end{align*} 
This implies that
\begin{align*}
y_n=(\lambda I-A_n)x_n=(\lambda I-\tilde A_n)\tilde x_n,\quad n\in\mathbb{N}.
\end{align*}
Now we infer
\begin{align*}
\lambda(x_n-\tilde x_n) & =A_n x_n -\tilde{A}_n \tilde x_n = A_n x_n -\tilde{A}_n \tilde x_n +\tilde A_n x_n - \tilde A_n x_n \\
& = (A_n-\tilde A_n)x_n +\tilde A_n(x_n - \tilde x_n)
\end{align*}
and
\begin{align*}
(\lambda I-\tilde A_n)(x_n-\tilde x_n)=(A_n-\tilde A_n)x_n.
\end{align*}
Since $((A_n -\tilde A_n) x_n)_n\in \mathcal{I}_X$ one obtains $((\lambda I-\tilde A_n)(x_n-\tilde x_n))_n\in \mathcal{I}_X.$ Applying $(R(\lambda,\tilde A_n))_n$ and using (G4) one obtains
\begin{align*} 
(x_n)_n \sim (\tilde x_n)_n \quad \Leftrightarrow \quad (R(\lambda,A_n)y_n)_n \sim (R(\lambda,\tilde A_n)y_n)_n.
\end{align*}
So one obtains $(R(\lambda,A_n))_n\simeq (R(\lambda,\tilde A_n))_n,$ $\lambda\in (\omega,\infty).$

\item[$(\Leftarrow)$] Now, let $(R(\lambda,A_n))_n\simeq(R(\lambda,\tilde A_n))_n,$ $ \lambda\in (\omega,\infty).$ Clearly, $D_{A_n}=D_{\tilde A_n}=D,$ $n\in\mathbb{N},$ since  $ {\rm range\;}R(\lambda,A_n)={\rm range\;}R(\lambda,\tilde A_n)=D,$ $n\in\mathbb{N}.$ Corollary \ref{cor D=R} implies $\mathbf{D}_A=\mathbf{D}_{\tilde{A}}.$

Finally, let us show that (GE3) holds.  Let $(x_n)_n\in \mathbf{D}_A,$ be given and denote by $(y_n)_n = (A_n x_n)_n,\;(\tilde y_n)_n = (\tilde A_n x_n)_n\in  \mathcal{E}_X^M.$ Then, for $\lambda>\omega,$
\begin{align*}
& (\lambda I-A_n)x_n=\lambda x_n - y_n \quad \Rightarrow \quad x_n = \lambda R(\lambda, A_n)x_n - R(\lambda, A_n)y_n\\
& (\lambda I-\tilde A_n)x_n=\lambda x_n - \tilde y_n \quad \Rightarrow \quad x_n = \lambda R(\lambda, \tilde A_n)x_n - R(\lambda, \tilde A_n)\tilde y_n.
\end{align*}
Next, for $\lambda>\omega,$
\begin{align*}
R(\lambda, \tilde A_n)\tilde y_n-R(\lambda, A_n) y_n & = \lambda R(\lambda,\tilde{A}_n)x_n - \lambda R(\lambda, A_n)x_n,\\
&=\left( R(\lambda, \tilde A_n)- R(\lambda, A_n)\right)(\lambda x_n).
\end{align*}
This relation and the assumption $(R(\lambda,A_n))_n\simeq(R(\lambda,\tilde A_n))_n,$ $\lambda>\omega$ imply that  $(R(\lambda, \tilde A_n)\tilde y_n)_n\sim (R(\lambda, A_n) y_n)_n.$ On the other hand,  since 
\begin{align*}
R(\lambda, \tilde A_n)\tilde y_n-R(\lambda, A_n) y_n & =R(\lambda, \tilde A_n)\tilde y_n-R(\lambda, A_n) y_n \pm R(\lambda,\tilde A_n)y_n\\
& = R(\lambda,\tilde A_n)(\tilde y_n-y_n)+(R(\lambda,\tilde A_n)-R(\lambda,A_n))y_n,
\end{align*}
we conclude that
\begin{align*}
(R(\lambda,\tilde A_n)\tilde y_n)_n \sim (R(\lambda,\tilde A_n) y_n)_n, \quad\lambda\in(\omega,\infty).
\end{align*}
Thus,  Lemma \ref{N u N} implies
\begin{align*}
((\lambda I- \tilde A_n)R(\lambda,\tilde A_n)\tilde y_n)_n \sim ((\lambda I- \tilde A_n)R(\lambda,\tilde A_n) y_n)_n \quad \Leftrightarrow \quad (\tilde y_n)_n \sim (y_n)_n.
\end{align*}
This means $ (A_n )_n \simeq (\tilde A_n)_n.$
\qedhere
\end{enumerate}
\end{proof}

\subsection{Relations between generators and g.e.i.s.'s}

We define the relation of equivalence for g.e.i.s.'s in the sense of  association: 

\begin{definition}\label{sg rel}
Let $(S_n)_n$ and $(\tilde S_n)_n$ be g.e.i.s.'s determined by $(A_n)_n,(\tilde A_n)_n\in \mathcal{A}^{M}_4.$  Then $(S_n)_n \simeq (\tilde S_n)_n$ if $((S_n-\tilde S_n)x_n)_n\in \mathcal{I}_{\rm exp}([0,\infty);X)$ for any $(x_n)_\in\mathcal{E}^M_X$ and the sequences of resolvents  $(R(\lambda,A_n))_n$ and $(R(\lambda,\tilde A_n))_n$ satisfy (RE1) and (RE2).
\end{definition}

\begin{Theorem}\label{T sg ka ig} 
Assume $(S_n)_n \simeq (\tilde S_n)_n.$ Then their infinitesimal generators  satisfy  $(A_n)_n\simeq (\tilde A_n)_n.$
\end{Theorem}

\begin{proof}
Let us prove that $((R(\lambda,A_n)-R(\lambda,\tilde A_n))x_n)_n\in \mathcal{I}_{X}$ for any $(x_n)_n\in\mathcal{E}_X^M.$ By Proposition  \ref{prop 1}, this implies that $(A_n)_n \simeq (\tilde{A}_n)_n.$ Let $\lambda>\omega$ and $(x_n)_n\in\mathcal{E}_X^M$ be fixed.  Then,
\begin{align*}
\left\|(R(\lambda,A_n)-R(\lambda,\tilde A_n))x_n\right\|_{X} & = \left\|\lambda \int_0^\infty e^{-\lambda t}\left(S_n(t)-\tilde{S}_n(t) \right)x_n\,dt \right\|_{X}\\
& \leq \sup_{t\geq 0}\|e^{-\omega t}(S_n(t)-\tilde{S}_n(t))x_n \|_{X} \left|\lambda \int_0^\infty e^{-(\lambda-\omega)t} dt \right| \\
& \leq C\sup_{t\geq 0}\|e^{-\omega t}(S_n(t)-\tilde{S}_n(t))x_n \|_{X}\to 0,\quad n\to\infty.
\end{align*}
This completes the proof.
\end{proof}

Next, we introduce sets of sequences of strong infinitesimal generators denoted by $\mathcal{A}^{M}_5.$ Denote by (G5) assumption \eqref{G5}, i.e.
\begin{itemize}
\item[(G5):] $\sup_{\rm Re\,\lambda>\omega}\| \lambda^b R(\lambda,A_n)\|_{\mathcal{L}(X)} \leq M_n,$ $n\in\mathbb{N},$ for some  $b>0$ and $(M_n)_n\in \mathcal{E}^M_\mathbb{R}.$
\end{itemize}

We say that a sequence $(A_n)_n$ is a sequence of strong infinitesimal generators,  that is $(A_n)_n\in\mathcal{A}^{M}_5$,  if $(A_n)_n\in \mathcal{A}^{M}_4$ and (G5) holds. Note that if $(A_n)_n$ satisfies (G5) then it also satisfies condition (G3).

Let us introduce a relation of equivalence in $\mathcal{A}^{M}_5$ as follows:  $(A_n)_n\simeq_0(\tilde A_n)_n$ if $(A_n)_n\simeq(\tilde A_n)_n$ and
\begin{enumerate}
\item[(GE4)]  there exists $b>0$ such that $\sup\limits_{\rm Re\,\lambda>\omega}\left\| \lambda^b \left(R(\lambda,A_n)-R(\lambda,\tilde A_n)\right)x_n\right\|_{X}\rightarrow 0,$ $n\to \infty,$ for all $(x_n)_n\in\mathcal{E}^M_X.$
\end{enumerate}

\begin{Theorem}
If $(A_n)_n\simeq_0(\tilde A_n)_n$ in $\mathcal{A}^{M}_5,$ then $(S_n)_n\simeq(\tilde S_n)_n.$
\end{Theorem}

\begin{proof}
Assumption (G5) and Theorem \ref{th 1} imply that $(A_n)_n$ and $(\tilde A_n)_n$ generate g.e.i.s.'s $(S_n)_n$ and $(\tilde S_n)_n,$ respectively.  Next, as in the proof of Theorem 2.5.1 in \cite{ABHN 2011},
\begin{equation*}
\|(S_n(t)-\tilde{S}_n(t))x_n\|_{X}\leq \frac{M_n e^{\alpha t}}{\pi}\int_R^\infty \frac{dr}{r^{1+b}}+\frac{M_n e^{\alpha t} }{\pi R^b}\int_0^{\pi/2} e^{Rt\cos\theta} \,d\theta,\; n\in\mathbb{N},\; t>0,
\end{equation*}
where
\begin{equation*}
M_n=\sup_{{\rm Re} \lambda>\omega}\|\lambda^b(R(\lambda,A_n)-R(\lambda,\tilde A_n))x_n\|_{X}.
\end{equation*}
Now assumption (GE4) above gives that $M_n\to 0,$ $n\to\infty.$ This finishes the proof.
\end{proof}

We recall a one more known assumption for integrated semigroups (called theoretically important, in \cite{N}, p.128).  Let a sequence $(A_n)_n\in\mathcal{A}^{M}_4$ be such that the common  domain $D$ of operators $A_n,$ $n\in\mathbb{N},$ given in (G1), is dense in $X$ ($\overline D=X$).  Assume
\begin{align}\label{d1}
\sup_{k\in \mathbb{N}_0}\sup_{\lambda>\omega}\| (\lambda-\omega)^{k+1}(R(\lambda,A_n)/\lambda)^{(k)}/k!\|_{\mathcal{L}(X)}\leq M_n, \quad n\in\mathbb{N},
\end{align}
for some $(M_n)_n\in\mathcal{E}^M_{\mathbb{R}}. $
Then, by Theorem 3.3.2 in \cite{ABHN 2011},  the sequence $(A_n)_n$ generate a sequence of exponentially bounded integrated semigroups $(S_n)_n$ such that $\|S_n(t)\|_{\mathcal{L}(X)} \leq M_n e^{\omega t};$ so $(S_n)_n\in \mathcal{C}^M_{\rm exp}([0,\infty);\mathcal{L}(X)).$
 
This classical result implies the following assertion.
 
\begin{Proposition}
Let $(A_n)_n,(\tilde A_n)_n\in\mathcal{A}^{M}_4$ satisfy \eqref{d1},  so that the common domain of $A_n,$ $\tilde A_n,$ $n\in\mathbb{N},$ $D$ is dense in $X.$ Assume that for all $(x_n)_n\in\mathcal{E}^M_X$
\begin{align}\label{d2}
\sup_{k\in \mathbb{N}_0}\sup_{\lambda>\omega}\| (\lambda-\omega)^{k+1}(((R(\lambda,A_n)-R(\lambda,\tilde A_n))/\lambda)^{(k)}x_n )/k!\|_{X}\rightarrow 0, \  n\rightarrow \infty.
\end{align}
Then  $(S_n)_n\simeq (\tilde S_n)_n.$
\end{Proposition}

\begin{proof}
We know, by Theorem 3.3.2 in \cite{ABHN 2011}, that $(A_n)_n$ and $(\tilde A_n)_n$ generate g.e.i.s.'s $(S_n)_n$ and $(\tilde S_n)_n.$  Let $(x_n)_n\in\mathcal{E}^M_X$ and let $m_n,$ $n\in\mathbb{N}$ denote, in \eqref{d2},
\begin{align*}
\sup_{k\in \mathbb{N}_0}\sup_{\lambda>\omega}\| (\lambda-\omega)^{k+1}(((R(\lambda,A_n)-R(\lambda,\tilde A_n))/\lambda)^{(k)}x_n )/k!\|_{X}\leq m_n \rightarrow 0, \quad n\rightarrow \infty.
\end{align*}
Now, Theorem 2.4.2 in \cite{ABHN 2011} implies that there exists $(g_n)_n\in(L^1_{\rm loc}([0,\infty),X))^\mathbb{N}$ so that
\begin{align*}
\frac{(R(\lambda,A_n)-R(\lambda,\tilde{A_n}))x_n}{\lambda}=\int_0^\infty e^{-\lambda t}g_n(t)\ dt,\quad \lambda>\omega,
\end{align*}
and $\|g_n(t)\|_X\leq m_ne^{\omega t},$ $t\geq 0,$ $n\in\mathbb{N}.$

Also, we know that there exist $(H_n)_n$ and $(\tilde H_n)_n$ in $(L^1_{\rm loc}([0,\infty),X))^\mathbb{N}$ determined by $R(\lambda,A_n)$ and $R(\lambda,\tilde A_n)$ so that 
\begin{align*}
\frac{R(\lambda,A_n)x_n}{\lambda}-\frac{R(\lambda,\tilde{A_n})x_n}{\lambda}=\int_0^\infty e^{-\lambda t}(H_n(t)-\tilde H_n(t))\ dt=\int_0^\infty e^{-\lambda t}g_n(t)\ dt.
\end{align*}
Thus, by the uniqueness of the Laplace transform, we have $\|H_n(t)-\tilde H_n(t)\|_X\leq m_n e^{\omega t},$ $n\in\mathbb{N}.$ With $H_n(t)=S_n(t)x_n$ and $\tilde H_n(t)=\tilde S_n(t)x_n,$ $n\in \mathbb{N},$ we complete the assertion.
\end{proof}

\subsection{Perturbations}

We finish the paper with the result concerning the perturbations. It directly follows from the corresponding one in  \cite{KH}, Section 3. The proof is omitted.

Let $(A_n)_n\in \mathcal{A}^{M}_5,$  be a sequence of infinitesimal generators of g.e.i.s $(S_n)_n\in \mathcal{C}^M_{\text{exp}}([0,\infty);\mathcal{L}(X)).$ Let $(B_n)_n\in \mathcal{E}^M_{\mathcal{L}(\bar D)}$ so that $||B_n||_{\mathcal{L}(\bar D)}\leq C, $ $n\in \mathbb N$ for some $C>0$ ($\bar D$ is the closure of $D=D_{A_n},$ $ n\in\mathbb{N}$). Assume that there exists $\lambda_0$ such that
\begin{align*}
B_nR(\lambda,A_n)=R(\lambda,A_n)B_n, \quad \lambda>\lambda_0, \quad n\in\mathbb N.
\end{align*}
Let (as in \cite{KH})
\begin{equation}\label{eq tacka}
S^{B_n}_n (t)=e^{tB_n}S_n (t)-B_n \int_0^t e^{sB_n}S_n (s)\ ds,\quad t>0,\;\;  n\in\mathbb{N}.
\end{equation}
Then we have the next adaptation of Proposition 3.1 in \cite{KH}.

\begin{Proposition} Let $(A_n)_n\in\mathcal{A}^{M}_5$ and $(B_n)_n\in \mathcal{E}^M_{\mathcal{L}(\bar D)}$ satisfy all assumptions given above.
\begin{enumerate}
\item $(A_n +B_n)_n\in \mathcal{A}^{M}_5. $  It is a sequence of infinitesimal generators of $(S^{B_n}_n)_n$ given by \eqref{eq tacka} and $(S_n^{B_n})_n$ is g.e.i.s.
\item Let $(C_n)_n\in \mathcal{I}_{\mathcal{L}(\bar D)}$ such that $C_nR(\lambda,A_n)=R(\lambda,A_n)C_n,$ $\lambda>\lambda_0,$ $n\in\mathbb{N},$ and $\tilde B_n=B_n+C_n,$ $n\in\mathbb{N}.$ Then ${(S_n^{B_n})}_n \simeq (S_n^{\tilde B_n})_n$ and ${(A_n +B_n)}_n \simeq_0{(A_n+\tilde B_n)}_n.$
\item Let $(A_n)_n,{(\tilde A_n)}_n\in \mathcal{A}^{M}_5.$ Then 
\begin{align*}
(A_n)_n \simeq_0 {(\tilde A_n)_n}\quad  \Rightarrow\quad  
{(S^{B_n}_n)_n} \simeq (\tilde S^{B_n}_n)_n.
\end{align*}
\end{enumerate}
\end{Proposition}

With this proposition we can  construct associated infinitesimal generators which produce associated g.e.i.s.'s.

For example, continuing with the example concerning the differential o\-pe\-ra\-tor given in the last part of Section \ref{examples}, let $X=L^p(\mathbb R), $ $p\in(1,\infty)$, $A=P(D)=a_0+a_1D+a_2D^2,$ $p(i\xi)=\sum_{j=0}^2 a_j(i\xi)^j$ and make perturbation so that $A_n= P_n (D)=a_0+1/n +a_1D+(a_2+1/n)D^2$ and $p_n(i\xi)=(a_0+1/n)+a_1(i\xi)+(a_2+1/n)(i\xi)^2.$

We have that $D(A)=D(A_n)=W^{2,2}(\mathbb R)$ and $A-A_n=(1+D^2)/n,$ $n\in\mathbb{N}.$ So these sequences are associated and 
\begin{align*}
S(t)f=\mathcal{F}^{-1}\left(\int_0^t e^{p(i\xi)s}ds\right)\ast f,\quad  S_n(t)f=\mathcal{F}^{-1}\left(\int_0^t e^{p_n(i\xi)s}ds\right)\ast f,
\end{align*}
$f\in X,\ n\in\mathbb{N},$ and we have that for given $f\in X$
\begin{align*}
\sup_{t>0}\|S_n (t)f-S(t)f\|_{X}\to 0,\quad n\to \infty.
\end{align*}

\section*{Acknowledgment}

The paper is supported by the following projects and grants: project F10 of the Serbian Academy of Sciences and Arts, project 142-451-2384 of the Provincial Secretariat for Higher Education and Scientific Research and projects 451-03-47/2023-01/200125 and 337-00-577/2021-09/46 of the Ministry of Education, Science and Technological Development of the Republic of Serbia.


\end{document}